\documentclass[10pt]{amsart}
\usepackage{amssymb,amsfonts}
\usepackage[normalem]{ulem}
\usepackage[all,arc]{xy}
\usepackage{enumitem}
\usepackage{mathrsfs}
\usepackage{ stmaryrd }
\usepackage{url}
\usepackage{tikz-cd}

\usepackage{fancyhdr}
\newcommand{\dia}[1]{{\langle #1 \rangle}}

\newcommand{\ttmat}[4]{\left( \begin{array}{cc}
#1 & #2 \\
#3 & #4
\end{array}
\right)}
\newcommand{\sm}[4]{\ensuremath{\big(\begin{smallmatrix}#1 & #2 \\ #3 & #4\end{smallmatrix}\big)}}

\newcommand{\up}[1]{^{[#1]}}

\newcommand{\isoto}{\xrightarrow{\sim}}

\newcommand{\onto}{\twoheadrightarrow}
\newcommand{\Tr}{\mathrm{Tr}}
\newcommand{\cC}{\mathcal{C}}
\newcommand{\ad}{\mathrm{ad}}
\newcommand{\red}{\mathrm{red}}

\newcommand{\F}{\mathbb{F}}

\newcommand{\m}{\mathfrak{m}}

\newcommand{\sO}{\mathcal{O}}
\newcommand{\p}{\mathfrak{p}}
\newcommand{\Q}{\mathbb{Q}}

\newcommand{\T}{\mathbb{T}}

\newcommand{\Z}{\mathbb{Z}}

\DeclareMathOperator{\Ann}{Ann}

\DeclareMathOperator{\End}{End}

\DeclareMathOperator{\Frob}{Frob}
\DeclareMathOperator{\Gal}{Gal}

\DeclareMathOperator{\GL}{GL}
\DeclareMathOperator{\Hom}{Hom}

\DeclareMathOperator{\Ind}{Ind}

\DeclareMathOperator{\tr}{tr}

\newcommand{\bT}{\mathbb{T}}

\newcommand{\lb}{\llbracket}
\newcommand{\rb}{\rrbracket}

\usepackage{hyperref}
\hypersetup{
    colorlinks=true,
    linkcolor=blue,
    urlcolor=blue, 
    citecolor= blue}
\usepackage[noabbrev, capitalize, nameinlink]{cleveref}

\newtheorem{theorem}{Theorem}[section]
\newtheorem{corollary}[theorem]{Corollary}
\newtheorem{proposition}[theorem]{Proposition}
\newtheorem{lemma}[theorem]{Lemma}

\theoremstyle{definition}
\newtheorem{definition}[theorem]{Definition}

\newtheorem{example}[theorem]{Example}

\theoremstyle{remark}
\newtheorem{remark}[theorem]{Remark}

\title{The Eisenstein ideal at prime-square level has  constant rank}
\author{Jaclyn Lang}

\author{Preston Wake}
\begin{document}

\maketitle

\begin{abstract}
Let $N$ and $p$ be prime numbers with $p \ge 5$ such that $p \mid \mid (N+1)$.  
In a previous paper, we showed that there is a cuspform $f$ of weight $2$ and level $\Gamma_0(N^2)$ whose $\ell$-th Fourier coefficient is congruent to $\ell + 1$ modulo a prime above $p$ for all primes $\ell$.  In this paper, we prove that this form $f$ is unique up to Galois conjugacy, and the extension of $\Z_p$ generated by the coefficients of $f$ is exactly $\Z_p[\zeta_p+\zeta_p^{-1}]$.  We also prove similar results when a higher power of $p$ divides $N+1$.
\end{abstract}

\section{Introduction}

Let $p$ be a prime number and let $\bar{\rho}: \Gal(\overline{\Q}/\Q) \to \GL_2(\bar\F_p)$ be a modular residual Galois representation.  How many different Hecke eigenforms $f$ give rise to $\bar{\rho}$, and what can be said about the $p$-adic field $\Q_p(f)$ generated by the Hecke eigenvalues of $f$?  One can fine tune this question by constraining the various parameters involved. For instance, if one fixes the level of $f$ but allows the weight to vary, then Buzzard, motivated by conjectures about slopes of modular forms, asked whether the degrees $[\Q_p(f):\Q_p]$ are bounded \cite[Question 4.4]{buzzard}. In the case $p=2$ with level $1$, Buzzard even suggested a bound of $2$ on $[\Q_2(f):\Q_2]$.  We know of very little progress on this question in the twenty years since Buzzard asked it\footnote{While preparing this article, we learned of recent work by Kimball Martin and Anna Medvedovsky giving examples of level-one $f$ where $[\Q_2(f):\Q_2]>2$.  There does not seem to be a consensus among experts about the question of boundedness.}.

In this paper, we consider a question orthogonal to Buzzard's: we fix the representation $\bar\rho=\omega\oplus 1$, where $\omega: G_\Q \to \F_p^\times$ is the mod-$p$ cyclotomic character.  We are interested in Hecke eigenforms of fixed weight $2$ that give rise to $\bar{\rho}$, but we allow the level to vary (with certain constraints).  There is another way to phrase this in terms of Hecke algebras: there is a localization $\bT$ of the Hecke algebra whose minimal prime ideals correspond to Galois-conjugacy classes of such eigenforms, and we are interested in the number and degree of these minimal primes.

If the level is constrained to be prime, Mazur asked about the rank of the Hecke algebra $\T$ \cite[p.~140]{Mazur77}\footnote{Calegari and Emerton \cite{CE05} first pointed out the parallels between Mazur's question and Buzzard's.}.  In this prime-level case, the degrees of $\Q_p(f)$ (and their sum --- the rank of $\bT$) have a great deal of arithmetic significance.  They have been studied using modular symbols by Merel \cite{Merel96} and Lecouturier \cite{Lecouturier21},  where they are shown to be related to special values of equivariant $L$-functions.  Using Galois representations, Calegari--Emerton \cite{CE05} and Wake--Wang-Erickson \cite{WWE20} show that these ranks are related to class groups and Massey products in Galois cohomology.  In numerical examples, the most common scenario is that there is a unique cusp form giving rise to $\bar{\rho}$ and its Hecke field is $\Q_p$, but this is certainly not always the case.  See the tables in \cite[p.~40]{Mazur77} and \cite[\S~1.6]{WWE20} for data about the rank and irreducible components of $\T$ in small prime level.  It is not known whether the degrees of the $\Q_p(f)$ are bounded independently of the level $N$. Heuristics given in \cite{WWE20} suggest that, given a prime $N$ with $p \mid\mid(N-1)$, the probability that there is a form $f$ of level $N$ such that $[\Q_p(f):\Q_p]=d$ is $\frac{p-1}{p^d}$. This accounts for the numerical evidence that the degrees $[\Q_p(f): \Q_p]$ are usually small but suggests that they are unbounded.

In this paper, we consider the same representation $\bar\rho=\omega\oplus 1$ and same weight $2$, but we vary the level over \emph{squares} of primes $N$ such that $p | (N+1)$.  For such primes $N$, Mazur's results imply that there are no newforms $f$ of level $\Gamma_0(N)$ giving rise to $\bar\rho$.  However, in our previous work, we show that there is a newform of level $\Gamma_0(N^2)$ with residual representation $\bar\rho$ \cite[Theorem B]{us}.  Our main result in this paper is that in this case we can compute the Hecke algebra $\T$ explicitly.  More precisely we have the following theorem. 

\begin{theorem}
\label{thm:main}
Let $N, p \ge 5$ be prime numbers such that $N \equiv -1 \bmod p$, and let $r \ge 1$ be the $p$-adic valuation of $N+1$.  Let $\T$ be the Hecke algebra parametrizing modular forms of level $\Gamma_0(N^2)$ and weight $2$ with mod-$p$ residual representation $\omega \oplus 1$. Let $\Delta = \F_{N^2}^{\times, p\mathrm{-part}}$, a cyclic group of order $p^r$,  let $\Lambda=\Z_p[\Delta]$, and let $\Lambda^+ \subset \Lambda$ be the subring fixed under the involution given by inversion on $\Delta$.  Then there is a canonical isomorphism of $\Z_p$-algebras $\Lambda^+ \isoto \bT$ sending the augmentation ideal of $\Lambda^+$ to the Eisenstein ideal of~$\bT$.
\end{theorem}

Since minimal prime ideals of $\bT$ correspond to (Galois-conjugacy classes of) eigenforms, the theorem allows us to answer all of the questions asked at the start of this introduction: the fields $\Q_p(f)$ correspond to the fraction fields of minimal primes of $\Lambda^+$.  The situation is particularly simple when $p \mid\mid (N+1)$ (i.e.~ when $r=1$), in which case there is only one minimal prime of $\Lambda^+$ other than the augmentation ideal, and this prime has residue ring $\Z_p[\zeta_p+\zeta_p^{-1}]$, so the theorem yields the following.

\begin{corollary}
\label{cor:main}
Let $N, p \ge 5$ be prime numbers such that $p \mid\mid (N+1)$. Then there is a cuspidal eigenform $f$ of level $\Gamma_0(N^2)$ and weight $2$ with coefficients in $\Z_p[\zeta_p+\zeta_p^{-1}]$ such that
\begin{equation}
\label{eq:eis con}
a_\ell(f) \equiv 1 + \ell \pmod{\p}
\end{equation}
for all prime numbers $\ell$, where $\p$ is the maximal ideal of $\Z_p[\zeta_p+\zeta_p^{-1}]$.  Moreover, this form is the unique (up to Galois-conjugacy) cuspform satisfying \eqref{eq:eis con}.
\end{corollary}
Note that this implies that the fields $\Q_p(f)$ are independent of $N$.  Contrast this with the prime-level case, where the heuristic suggests that the degrees $[\Q_p(f) : \Q]$ are unbounded as one varies over primes $N$ with $p \mid\mid (N-1)$.

Of course, there is a variant of \cref{cor:main} when $r>1$: in that case,  there are $r$ Galois-conjugacy classes, and the coefficient rings are $\Z_p[\zeta_{p^i}+\zeta_{p^i}^{-1}]$ for $i=1,\dots,r$.

\subsection{Outline of the paper}
The proof of \cref{thm:main} uses Galois deformation theory, and takes up most of the paper.  We sketch the proof here,  indicating in which sections the steps take place. Let $R$ be the pseudodeformation ring of $\overline{\rho}$ parametrizing deformations that have fixed determinant and that are unramified outside $Np$ and finite-flat at $p$.  (The theory of pseudodeformations is reviewed in \cref{sec:pseudodef}.)   As usual,  there is a surjection $R \onto \T$.  We define the `pseudo-minimal' quotient $R^\mathrm{pseudo-min}$ of $R$ corresponding to deformations whose trace equals the trace of the trivial representation on inertia-at-$N$.  In other words, $R^\mathrm{pseudo-min}$ parameterizes representations for which the \emph{semisimplification} of the restriction to inertia-at-$N$ is trivial. This includes representations that are unramified at $N$, but also representations that are Steinberg at $N$.  However, Mazur's results imply that there are no cuspforms of level $\Gamma_0(N)$ 
that are congruent to the Eisenstein series, so one would expect that there are no representations that are Steinberg at $N$.  In  \cref{sec:reduce to local}, we prove that this is true: $R^\mathrm{pseudo-min} = \Z_p$. This key result shows that $R$ is entirely determined by the local behavior at $N$.
In \cref{sec:local}, we define a local-at-$N$ pseudodeformation ring $R_N$, and prove that all local deformations come from inducing a character of $G_{\Q_{N^2}}$,  which gives an isomorphism $R_N \cong \Lambda^+$.  Together with $R^\mathrm{pseudo-min} =\Z_p$, this gives surjections $\Lambda^+ \onto R \onto \bT$.  To complete the proof, in \cref{sec:proof} we show that these surjections are isomorphisms using Wiles's numerical criterion,  applying our previous results \cite{us} to understand the congruence number.  Finally,  in \cref{sec:massey}, we indicate how our results are related to the Massey-products method of \cite{WWE20}.

\subsection{Acknowledgements}
We thank Shaunak Deo, Robert Pollack, Alice Pozzi, and Carl Wang-Erickson for helpful conversations.  J.L.~was supported by NSF grant DMS-2301738 and P.W.~was supported by NSF CAREER grant DMS-2337830.

\section{Pseudodeformations}
\label{sec:pseudodef}
In this section, we review the aspects of deformation theory of pseudorepresentations that we will need in the next section. There are no new results in this section; it is a digest of material from many sources, including \cite{BC2009,Chenevier,BellaichePseudo,CarlThesis,WWE4}  . 

\subsection{Pseudorepresentations}
The concept of a pseudorepresentation came about to codify the formal properties of the trace (or, more generally, the characteristic polynomial) of a representation.  The first definition of pseudorepresentation was made by Wiles \cite{wiles} for 2-dimensional representations and was later generalized by Taylor \cite{Taylor}, Rouquier \cite{rouquier}, and Chenevier \cite{Chenevier}. We will use Chenevier's version, which he calls ``determinants''.

Chenevier's notion of pseudorepresentation mimics the properties of the determinant.  For a commutative ring $A$, the determinant map
$\det: M_n(A) \to A$
has many well-known properties: it is multiplicative,  in that $\det(xy)=\det(x)\det(y)$ and $\det(1)=1$, and has degree $n$, in that $\det(ax)=a^n\det(x)$ for $a \in A$ and $x,y \in M_n(A)$.  It is also a polynomial function in the entries of the matrix.  In particular, if $B$ is a commutative $A$-algebra, then one can also apply $\det$ to an element of the tensor product $M_n(A)\otimes_A B$ and obtain, in a natural way, an element of $B$.  In particular, taking $B=A[t]$ allows one to define the characteristic polynomial $\det(t-x) \in A[t]$ of a matrix $x$.

\begin{definition}
Let $A$ be a commutative ring and $E$ an $A$-algebra. A \emph{pseudorepresention of $E$ of degree $d$}, written $D:E \to A$, is a collection of maps $D_B: E \otimes_A B \to B$, one for each commutative $A$-algebra $B$, that are natural in $B$ and satisfy:
\begin{itemize}
\item $D_B(xy)=D_B(x)D_B(y)$ and $D_B(1)=1$,
\item $D_B(bx)=b^dD_B(x)$
\end{itemize}
for all $x,y \in E\otimes_A B$ and all $b \in B$. The map $D_A$ is abbreviated to $D$.

If $G$ is a group, then a \emph{pseudorepresentation of $G$ of degree $d$ over $A$}, written $D: G \to A$, is a pseudorepresentation of $A[G]$. 
\end{definition}

In this paper, we will be interested exclusively in degree-two pseudorepresentations and only in the case where $2$ is invertible in the ring $A$.  In this case, a pseudorepresentations have simpler description \cite[Example 1.8]{Chenevier}, as we now recall. For a $2$-dimensional pseudorepresentation $D:E \to A$, define the trace $\Tr_D:E \to A$ of $D$ by the formula
\[
\Tr_D(x)=D(x+1)-D(x)-1.
\]
It is a simple calculation to see that the familiar formula for the characteristic polynomial of a $2\times 2$-matrix holds:
\[
D_{A[t]}(t-x)=t^2-\Tr_D(x)t+D(x).
\]
For a degree-two pseudorepresention $D: G \to A$ of $G$,  the trace satisfies relations 
\begin{enumerate}
\item $\Tr_D(xy)=\Tr_D(yx)$ and $\Tr_D(1)=2$, and
\item
\label{item: pseudorep ident}
 $D(x)\Tr_D(x^{-1}y)-\Tr_D(x)\Tr_D(y)+\Tr_D(xy)=0$
\end{enumerate}
for all $x,y \in G$.
Conversely, if $D':G \to A^\times$ is a homomorphism and $T:G \to A$ is a function such that the pair $(D',T)$ satisfy (1) and (2), then the formula
\[
D_B(bx+y) = D'(x)b^2+(T(x)T(y)-T(xy))b+D'(y)
\] 
for $x,y \in G$ and $b\in B$,
defines a pseudorepresentation $D:G \to A$. In this way, one can think of a pseudorepresentation as the data of the functions $D$ and $\Tr_D$.  
Moreover, if $2$ is invertible in $A$, then one can recover $D_A$ from $\Tr_D$ using formula (2) as $D(x)=\frac{\Tr_D(x)^2-\Tr_D(x^2)}{2}$. 
\begin{remark}
\label{rem:continuity}
Thus far in the discussion, we have considered discrete groups and rings.  For topological groups and rings,  one considers \emph{continuous} pseudorepresentations $D:G \to A$, which amounts to requiring that the functions $D: G \to A$ and $\Tr_D: G \to A$ are continuous (see \cite[Section 2.30]{Chenevier}). 
If $D:G \to A$ is a continuous pseudorepresentation and $H \subseteq G$ is a dense subgroup, then $D$ is determined by its restriction to $H$ \cite[Example 2.31]{Chenevier}.  To simplify the discussion below, we will always assume that pseudorepresentations are continuous if we are using topological groups.
\end{remark}

\begin{example}
\label{ex:pseudorep from rep}
Let $G$ be a group and $A$ be a commutative ring. If $\rho:G \to \GL_2(A)$ is a homomorphism, then the pair of functions $(D,T)=(\det \circ \rho, \tr \circ \rho)$ is, of course,  a pseudorepresentation.  Moreover, if there is a subring $A' \subseteq A$ such that $D$ and $T$ both have images in $A'$, then $(D,T)$ defines a pseudorepresentation $D: G \to A'$. 
\end{example}

This example can be seen as one of the major advantages of pseudorepresentations, and is the purpose for which Wiles first used them. Note that it may not be true that there is a conjugate $\rho'$ of $\rho$ such that $\rho'$ has values in $\GL_2(A')$, and, even if such a $\rho'$ exists, it may not be unique up to  $\GL_2(A')$-conjugation. 

\subsection{Cayley--Hamilton algebras and generalized matrix algebras} Not every pseudorepresentation $D:G \to A$ comes from a true representation $\rho: A[G]\to M_2(A)$ as in \cref{ex:pseudorep from rep}.  However,  every pseudorepresentation does come from a representation $A[G] \to E$ for a special type of algebra $E$ called a \emph{Cayley--Hamilton algebra}.  In good situations, this Cayley--Hamilton algebra can be given the extra structure of a \emph{generalized matrix algebra}, which have many useful properties in common with usual matrix algebras.

\begin{definition} Let $A$ be a commutative ring and let $E$ be an $A$-algebra.
A degree-two pseudorepresentation $D: E \to A$ is called \emph{Cayley--Hamilton} if
\[
x^2-\Tr_D(x)x+D_B(x) =0
\]
for all $x \in E \otimes_A B$.  A pair $(E,D)$ of an $A$-algebra $E$ and a Cayley--Hamilton pseudorepresentation is called a \emph{Cayley--Hamilton algebra}.

If $G$ is a group, then a \emph{Cayley--Hamilton representation} of $G$ is a triple $(E,D,\rho)$ where $(E,D)$ is a Cayley--Hamilton algebra and $\rho: G \to E^\times$ is a group homomorphism. The composition $D \circ \rho : G \to A$ defines a pseudorepresentation $\psi(\rho)$ of $G$ over $A$ called \emph{the associated pseudorepresentation}.
\end{definition}

For example, the algebra $E=M_2(A)$ with the pseudorepresentation given by the determinant is Cayley--Hamilton, by the Cayley--Hamilton Theorem (whence the name).  A representation $\rho: A[G] \to M_2(A)$ gives rise to a Cayley--Hamilton representation, just as in \cref{ex:pseudorep from rep}.

\begin{definition}
Let $A$ be a commutative ring and let $E$ be an $A$-algebra that is finitely generated as an $A$-module. A (2-dimensional) \emph{generalized matrix algebra structure} on $E$ is the data of
\begin{itemize}
\item an idempotent element $e \in E$,
\item $A$-algebra isomorphisms $\phi:eEe \cong A$ and $\phi':e'Ee' \cong A$,  where $e'=1-e$,
\end{itemize}
such that the function $\tr: E \to A$ defined by
\[
\tr(x)=\phi(exe)+\phi'(e'xe')
\]
satisfies $\tr(xy)=\tr(yx)$ for all $x,y \in E$.

An $A$-algebra $E$ together with an generalized matrix algebra structure is called an $A$-GMA.
\end{definition}

An example of an $A$-GMA is the matrix algebra $E=M_2(A)$ with $e=\sm{1}{0}{0}{0}$ and the obvious isomorphisms $eM_2(A)e \cong A$ and $e'M_2(A)e' \cong A$.  In general, an $A$-GMA can be written in the form
\[
E = \ttmat{A}{B}{C}{A},
\]
where $B=eEe'$ and $C=e'Ee$ are sub-$A$-modules of $E$.  The multiplication can be written as
\begin{equation}
\label{eq:GMA mult}
\ttmat{a}{b}{c}{d} \ttmat{a'}{b'}{c'}{d'} = \ttmat{aa'+m(b,c')}{ab'+bd'}{ca'+dc'}{dd'+m(c,b')},
\end{equation}
where $m: B \times C \to A$ is the map $m(exe',e'ye)=\phi(exe'ye)$.  

Conversely, if $B$ and $C$ are two finitely generated $A$-modules and $m: B \otimes_A C \to A$ is an $A$-linear map satisfying certain properties, then defining $E:= \ttmat{A}{B}{C}{A}$ with the multiplication as in \eqref{eq:GMA mult} defines an $A$-GMA (see \cite[Section 1.3]{BC2009} and \cite[Example 3.1.7]{WWE4} for more precise statements).

\begin{example}
\label{ex:1-red}
If $A=k[x]/(x^n)$ for a field $k$, then there is a GMA $E$ given by
\[
E= \ttmat{A}{xA}{xA}{A},
\]
where $m: xA \times xA \to A$ is $m(ax,bx)=abx$.
\end{example}

\subsection{Pseudodeformation rings}
Let $G$ be a group and $\F$ be a finite field of characteristic $p$ and let $\bar D:G \to \F$ be a pseudorepresentation.  In this section, we discuss deformations of $\bar D$.  We assume that $G$ is profinite and satisfies Mazur's finiteness condition: for every open normal subgroup $H \subseteq G$, there are only finitely many continuous group homomorphisms $H \to \Z/p\Z$. For instance, $G$ could be the absolute Galois group of a local field or the Galois group of the maximal extension of a number field that is unramified outside a finite set.

Let $\cC$ be the category of complete local Noetherian $W(\F)$-algebras $(A,\m_A)$ with residue field $\F$.  For an object $A$ in $\cC$, a \emph{deformation} of $\bar D$ to $A$ is a pseudorepresentation $D: G \to A$ such that $D\otimes_A \F = \bar D$. 
The set-valued functor on $\cC$ sending $A$ to the set of deformations of $\bar D$ to $A$ is representable by a ring $(R_{\bar D},\m_{\bar D})$ in $\cC$ \cite[Proposition E]{Chenevier}. The resulting pseudorepresentation $D^u:G \to R_{\bar D}$ is called the universal pseudodeformation.

A Cayley--Hamilton representation $(E,D,\rho)$ of $G$ is said to have residual representation $\bar D$ if the associated pseudorepresentation $\psi(\rho)$ of $G$ is a deformation of $\bar D$. The collection of Cayley--Hamilton representations with residual representation $\bar D$ forms a category in a natural way, and this category has a universal object $(E_{\bar D},D^u_{E_{\bar E}},\rho^u)$, which is a Cayley--Hamilton algebra over $R_{\bar D}$ and whose associated pseudorepresentation is the universal pseudodeformation \cite[Proposition 3.6]{CarlThesis}.

Now assume that $\bar D = \chi_1 \oplus \chi_2$ for two distinct characters $\chi_1, \chi_2: G \to \F^\times$.  In this case, there is a natural generalized matrix algebra structure on $E_{\bar D}$, written as
\begin{equation}
\label{eq:EDbar GMA}
E_{\bar D} = \ttmat{R_{\bar D}}{B_{\bar D}}{C_{\bar D}}{R_{\bar D}}
\end{equation}
with the property that, if $\rho^u:G \to E_{\bar D}$ is written as $\rho^u(g)=\sm{a(g)}{b(g)}{c(g)}{d(g)}$, then $a(g) \equiv \chi_1(g) \pmod{\m_{R_{\bar D}}}$.  See \cite[Lemma 1.4.3]{BC2009} and \cite[Theorem 2.22]{Chenevier} for more details.

\subsection{Tangent spaces}
The (equicharacteristic) tangent space to a deformation functor is the set of first-order deformations (that is, deformations with values in the dual numbers); this set is naturally a vector space over the residue field. For a representation $\bar\rho: G \to \GL_2(\F)$,  this means looking at deformations $\rho: G \to \GL_2(\F[\epsilon]/(\epsilon^2))$. It is well-known that, in this case, the tangent space can be identified with the group cohomology $H^1(G,\ad(\bar\rho))$ (see 
\cite[Proposition 1, pg.~284]{MazurDef}, for instance).  The identification sends a cocycle $\phi \in Z^1(G,\ad(\bar\rho))$ to the deformation
\[
\rho_\phi =  (1+\phi \epsilon)\bar\rho: G \to \GL_2(\F[\epsilon]/(\epsilon^2)).
\]

The computation of the tangent space of a pseudodeformation ring is similar to this but is complicated by the fact that not all of these deformations alter the pseudorepresentation. For instance, if $\bar\rho=\sm{\chi_1}{0}{0}{\chi_2}$ for distinct characters $\chi_1$ and $\chi_2$, then there is an isomorphism of $G$-modules $\ad(\bar\rho)\cong \sm{\F}{\F(\chi_1\chi_2^{-1})}{\F(\chi_1^{-1}\chi_2)}{\F}$.  If $\phi \in Z^1(G, \sm{\F}{0}{0}{\F})$, then the deformation $\rho_\phi$ amounts to deforming the two characters $\chi_1$ and $\chi_2$ separately and does change the pseudorepresentation.  However, if $b \in Z^1(G, \F(\chi_1\chi_2^{-1})) \subset Z^1(G,\ad(\bar\rho))$, then $\rho_b= \sm{\chi_1}{\chi_2 b \epsilon}{0}{\chi_2}.$
This is a nontrivial deformation of $\bar\rho$, but, since the trace and determinant are unchanged, it is a trivial pseudodeformation.

To get a nontrivial pseudodeformation out of cocycles $b \in Z^1(G, \F(\chi_1\chi_2^{-1}))$ and $c \in  Z^1(G, \F(\chi_1^{-1}\chi_2))$ one has to assume more. Namely, if the cup product $b \cup c$ vanishes in $H^2(G,\F)$, then there is a cochain $\phi:G \to \F$ such that $d \phi = b\cup c$. There is also a cochain $\phi':G \to \F$ such that $d\phi'=c \cup b$, namely $\phi'=bc-\phi$. If, in addition, there is a cochain $b_1:G \to \F$ such that $db_1=b \cup \phi + \phi' \cup b$, then one can define a representation using this data by
\begin{equation}
\label{eq:rho bc tangent}
\rho_{b,c,\phi} =\ttmat{\chi_1+\chi_1\phi \epsilon}{\chi_2 (b+b_1 \epsilon)}{\chi_1 c \epsilon}{\chi_2+\chi_2\phi'\epsilon}.
\end{equation}
Note that this is \emph{not} a deformation of $\bar\rho$ as a representation,  since its residual representation is $\sm{\chi_1}{\chi_2b}{0}{\chi_2}$, but it is a \emph{pseudodeformation}.  Let
\begin{equation}
\label{eq:bc tangent}
D_{b,c,\phi}=\Tr(\rho_{b,c,\phi})=\chi_1+\chi_2+\epsilon(\chi_1\phi+\chi_2\phi')
\end{equation}
be the associated pseudorepresentation, and note that it involves $\phi$, $b$, and $c$, but not $b_1$. In fact, one can prove that $D_{b,c,\phi}$ defines a pseudodeformation even without assuming the existence of the cochain $b_1$ (this can be proven using the GMA of \cref{ex:1-red}).

An exact description of the tangent space of a pseudodeformation ring has been worked out beautifully by Bella\"iche in \cite{BellaichePseudo} and generalized by Wang-Erickson in \cite[Section 3.3]{CarlAinf}.  Let $\bar D: G \to \F$ be $\bar D= \chi_1 \oplus \chi_2$ for distinct characters $\chi_1$ and $\chi_2$. Let $\m_{\bar D}$ be the maximal ideal of $R_{\bar D}$ and let $\mathfrak{t}_{\bar D} = \Hom_\F(\m_{\bar D}/(p,\m_{\bar D}^2), \F)$ be the tangent space.  By \cite[Theorem A]{BellaichePseudo}, there is an exact sequence
\begin{equation}
\label{eq:tangent}
0 \to H^1(G,\F) \oplus H^1(G,\F) \to \mathfrak{t}_{\bar D} \to H^1(G, \chi_1\chi_2^{-1}) \otimes_\F H^1(G,\chi_1^{-1}\chi_2) \xrightarrow{\cup}  H^2(G,\F) \oplus H^2(G,\F).
\end{equation}
The subspace $H^1(G,\F) \oplus H^1(G,\F)$ corresponds to the \emph{reducible deformations} that deform $\chi_1$ and $\chi_2$ separately. For $b \in H^1(G, \F(\chi_1\chi_2^{-1}))$ and $c \in  H^1(G, \F(\chi_1^{-1}\chi_2))$ such that $b \cup c=0$, the corresponding element of $\mathfrak{t}_{\bar D}$ is exactly \eqref{eq:bc tangent}.

\subsection{Reducibility ideal}  
We now return to the situation of \eqref{eq:EDbar GMA}, so $\bar D = \chi_1 \oplus \chi_2$ for distinct characters $\chi_1, \chi_2: G \to \F^\times$.  We say that a deformation $D$ of $\bar D$ is \emph{reducible} if $D=\tilde{\chi_1} \oplus \tilde{\chi_2}$ for deformations $\tilde{\chi_i}$ of $\chi_i$.  The reducible deformations define a subfunctor of the pseudodeformation functor that is represented by a quotient $R_{\bar D}^\red$ of $R_{\bar D}$.  The kernel of the map is called the \emph{ideal of reducibility} $J_{\bar D} =\ker(R_{\bar D} \to R_{\bar D}^\red)$. 

The ring $R_{\bar D}^\red$ is fairly easy to understand: it can be identified with the tensor product of deformation rings of the characters $\chi_i$ (see \cite[Proposition 4.6.2]{WWE4}). The ideal of reducibility is related to the GMA-structure on $E_{\bar D}$ by a theorem of Bella\"iche and Chenevier: $J_{\bar D}$ is the image of the map $B_{\bar D} \otimes_{R_{\bar D}} C_{\bar D} \to R_{\bar D}$ defined by the GMA-structure \eqref{eq:EDbar GMA} (see \cite[Section 1.5.1]{BC2009}).  In particular, there is a surjective map
\begin{equation}
\label{eq:BC=J}
B_{\bar D} \otimes_{R_{\bar D}} C_{\bar D}  \onto J_{\bar D}.
\end{equation}
Moreover, certain quotients of the modules $B_{\bar D}$ and $C_{\bar D}$ can be understood using group cohomology.  Let  $R_{\bar D} \to A$ be a morphism in $\cC$, and let $\chi_{1,A}, \chi_{2,A} : G \to A^\times$ be the corresponding deformations of $\chi_1$ and $\chi_2$. Then there is an isomorphism
\begin{equation}
\label{eq:B dual}
\Hom_A(B_{\bar D} \otimes_{R_{\bar D}} A, A) \cong H^1(G,\chi_{1,A}\chi_{2,A}^{-1})
\end{equation}
by \cite[Theorem 1.5.6]{BC2009} and a similar isomorphism for $C_{\bar D}$ with the roles of $\chi_{1,A}$ and $\chi_{2,A}$ reversed. 
 
Taken together, these results can give a fairly clear picture of the structure of $R_{\bar D}$, especially when the cohomology groups $H^1(G,\chi_{1,A}\chi_{2,A}^{-1})$ and $H^1(G,\chi_{1,A}^{-1}\chi_{2,A})$ are small.

\begin{example}
Suppose that $H^1(G,\chi_{1}\chi_{2}^{-1})$ and $H^1(G,\chi_{1}^{-1}\chi_{2})$ are both $1$-dimensional $\F$-vector spaces. Then \eqref{eq:B dual} and Nakayama's lemma imply that $B_{\bar D}$ and $C_{\bar D}$ are both cyclic $R_{\bar D}$-modules. By \eqref{eq:BC=J}, this implies that $J_{\bar D}$ is a principal ideal.  
\end{example}

To see how this compares to the tangent space sequence \eqref{eq:tangent}, consider the reduction of \eqref{eq:BC=J} modulo $\m_{\bar D}$:
\[
B_{\bar D}/\m_{\bar D} B_{\bar D} \otimes_\F C_{\bar D}/\m_{\bar D} C_{\bar D} \onto J_{\bar D}/\m_{\bar D} J_{\bar D} \onto J_{\bar D}/(p,\m_{\bar D}^2) \subseteq \m_{\bar D} /(p,\m_{\bar D}^2).
\]
Taking the $\F$-dual of this composite map and using \eqref{eq:B dual} gives a map
\[
\mathfrak{t}_{\bar D} \to H^1(G, \chi_1\chi_2^{-1}) \otimes_\F H^1(G,\chi_1^{-1}\chi_2)
\]
that equals the map in \eqref{eq:tangent}.
\subsection{Deformation conditions}
For applications to number theory, often one wants to consider deformations that satisfy certain conditions rather than the universal deformations considered thus far.  For instance,  one often wants to understand Galois representations that ``come from geometry", a condition can is usually expressed in terms of ramification and $p$-adic Hodge theory.  For deformations of representations, Ramakrishna worked out a theory for deformations with conditions \cite{Ramakrishna}, and this theory has been generalized to pseudodeformations \cite{WWE4}.

A \emph{deformation condition} on representations of a group $G$ is a full subcategory $\mathfrak{c}$ of the category of finite $\Z_p[G]$-modules that is closed under isomorphisms, submodules, quotient modules, and finite direct sums.  We think of this as a condition on modules, so we say that a module ``has $\mathfrak{c}$" if it is in $\mathfrak{c}$.  By definition, a pseudorepresention $D:G \to A$ of $G$ with values in a finite ring $A$ in $\cC$ \emph{has $\mathfrak{c}$} if there is a Cayley--Hamilton representation $(E,D_E,\rho)$ over $A$ such that the $\Z_p[G]$-module $E$ has $\mathfrak{c}$ and such that $D=D_E\circ \rho$.  A general ring $A$ in $\cC$ is a limit of finite rings, so the definition is extended to $A$ by taking limits.

With this definition,  the constructions and properties carried out in this section extend to pseudorepresentations with $\mathfrak{c}$. In particular, there are quotients $R_{\bar D, \mathfrak{c}}$ and $E_{\bar D, \mathfrak{c}}$ of $R_{\bar D}$ and $E_{\bar D}$ that parameterize deformations and having property $\mathfrak{c}$ \cite[Section 2.5]{WWE4}.  Moreover,  the analogs of \eqref{eq:BC=J} and \eqref{eq:B dual} hold with the $\mathfrak{c}$-versions, except that, in \eqref{eq:B dual}, the group cohomology $H^1(G,\chi_{1,A}\chi_{2,A}^{-1})$ needs to be replaced by the group $H^1_{\mathfrak{c}}(G,\chi_{1,A}\chi_{2,A}^{-1})$ of extensions 
\[
0 \to \chi_{1,A} \to \mathcal{E} \to \chi_{2,A} \to 0
\]
where $\mathcal{E}$ has $\mathfrak{c}$ \cite[Section 4.3]{WWE4}.  (This group $H^1_{\mathfrak{c}}$ is a natural generalization of the Bloch--Kato cohomology groups $H^1_e$, $H^1_f$, and $H^1_g$ \cite[Section 3, pg.~352]{BlochKato}.)

\section{Reduction to a local problem}
\label{sec:reduce to local}
In this section, we prove some important reductions toward the proof of \cref{thm:main}.  First we define the Hecke algebra $\bT$ and the pseudodeformation ring $R$ that are relevant to the problem and prove that there is a surjection $R \onto \bT$.  Then we analyze the tangent space of $R$ and use this to prove the key result: the pseudo-minimal quotient $R^\mathrm{pseudo-min}$ of $R$ is equal to $\Z_p$. 

\subsection{The Hecke algebra}
Let $M_2(\Gamma_0(N^2))$ denote the space of modular forms of weight $2$ and level $\Gamma_0(N^2)$ with integral coefficients, and let $\tilde{\bT}$ be the subring of $\End_\Z(M_2(\Gamma_0(N^2))$ generated by the Hecke operators $T_n$ for all $n$.  Let $I$, called the Eisenstein ideal, be the $\T$-ideal generated by $T_N$ and $T_\ell-\ell-1$ for all $\ell  \nmid N$.  Let $\m_{\bar\rho}$ be the maximal ideal generated by $I$ and $p$, and let $\bT$ be the completion of $\tilde{\bT}$ at $\m_{\bar\rho}$.  Finally, $\bT^0$ denotes the maximal quotient of $\bT$ that acts faithfully on the space of cusp forms, which is denoted $S_2(\Gamma_0(N^2))$.  By \cite[Theorem 2.8]{us}, the quotient $M_2(\Gamma_0(N^2))_{\m_{\bar\rho}}/S_2(\Gamma_0(N^2))_{\m_{\bar\rho}}$ is generated by a single Eisenstein series $E$, which is an eigenform for all $T_n$ whose $T_N$-eigenvalue is $0$.

\subsection{The pseudodeformation ring}
\label{subsec:Computation of R}
Let $G_{\Q,Np}$ be the Galois group of the maximal extension of $\Q$ that is unramified outside $\infty$, $N$, and $p$.  Fix embeddings of $\overline{\Q}$ into $\overline{\Q}_p$ and $\overline{\Q}_N$, and let $G_N, G_p \subset G_{\Q,Np}$ be the corresponding decomposition groups at $N$ and $p$.  Let $I_N \subset G_N$ and $I_p \subset G_p$ be their respective inertia groups. Let $\bar D: G_{\Q,Np} \to \F_p$ be the pseudorepresention $\omega \oplus 1$.  Let $\mathfrak{c}$ be the ``finite-flat'' condition; that is, $\mathfrak{c}$ is the category of finite $\Z_p[G_{\Q,Np}]$-modules $M$ such that there is a finite-flat group scheme $\mathcal{G}$ over $\Z_p$ such that $M \cong \mathcal{G}(\bar\Q_p)$ as $G_p$-modules. This is a deformation condition by \cite[Section 2]{Ramakrishna}. Let $R$ be the quotient of $R_{\bar D, \mathfrak{c}}$ parameterizing deformations that have determinant equal to the $p$-adic cyclotomic character, which we denote by $\epsilon$.  That is, $R$ is the quotient by the ideal generated by $D^u(\sigma)-\epsilon(\sigma)$ for all $\sigma \in G_{\Q,Np}$.   Abusing notation slightly, let $D^u:G \to R$ denote the composition of the universal deformation with $R_{\bar D} \onto R$.

\begin{lemma}
\label{lem:R to T}
There is a surjective $\Z_p$-algebra homomorphism $R \onto \bT$.
\end{lemma}
\begin{proof}
Since $\bT$ is known to be reduced, there is an injection $\bT \to \prod_{\p} \bT/\p$, where $\p$ ranges over the minimal primes of $\bT$.  There is one minimal prime given by the action of $\bT$ on the Eisenstein series $E$. The other minimal primes are the kernels of the maps $\bT \to \Q_p(f)$ for cuspidal eigenforms $f$ such that $\bar\rho_f = \omega \oplus 1$.  Let $S$ be the set of such cuspidal eigenforms.  Then there is an injection
\[
\bT \to \Z_p \times \prod_{f \in S} \Q_p(f),
\]
sending $T_n$ to $(a_n(E), (a_n(f))_{f \in S})$. We will identify $\bT$ with the image of this injection and construct a homomorphism $R \to \Z_p \times \prod_{f \in S} \Q_p(f)$ whose image is $\bT$.

For each $f \in S$,  the Galois representation $\rho_f$ defines a pseudorepresentation $D_f: G_{\Q,Np} \to \sO_f$ that deforms $\bar D$.  Since the level of $f$ is prime to $p$, $D_f$ satisfies the finite flat condition.  Also, the determinant of $D_f$ is $\epsilon$ since $f$ has weight 2 and trivial Nebentypus.  Hence $D_f$ defines a map $R \to \Q_p(f)$.  There is also a map $R \to \Z_p$ given by the pseudorepresentation $\epsilon \oplus 1$.  This defines a map
\[
\Phi: R \to \Z_p \times \prod_{f \in S} \Q_p(f).
\]
We have to show that the image of $\Phi$ is $\bT$.  For a prime $\ell \nmid Np$,  since $\Tr(\rho_f(\Frob_\ell))=a_\ell(f)$, the map $\Phi$ sends $\Tr_{D^u}(\Frob_\ell)$ to the image of $T_\ell$. Since the elements $\Tr_{D^u}(\Frob_\ell)$ topologically generate $R$ by Chebotarov density, this implies that the image of $\Phi$ equals the $\Z_p$-subalgebra of $\bT$ generated by $\{T_\ell \colon \ell \nmid Np \text{ prime}\}$.  It remains to show that this subalgebra contains $T_N$ and $T_p$. 
Since every $f \in S$ is new at $N^2$,  it follows that $a_N(f)=0$.  But $a_N(E)=0$ as well, so $T_N=0$ in $\bT$.
Finally,  $T_p$ is in the subalgebra generated by $\{T_\ell \colon \ell \nmid Np \text{ prime}\}$ by the ordinary property. Indeed, since each of the forms $f \in S$ is congruent to $E$ modulo $p$,  and $E$ is ordinary at $p$, each $f\in S$ is ordinary at $p$.  It follows that for $\sigma \in G_p$,
\[
\Tr(\rho_f)(\sigma) = \epsilon(\sigma)\lambda(a_p(f))(\sigma) + \lambda(a_p(f))^{-1}(\sigma),
\]
where $\lambda(x)$ is the unramified character of $G_p$ sending $\Frob_p$ to $x$.  If $\tau \in I_p$ is an element such that $\epsilon(\tau) \not \equiv 1 \mod{p}$, then
\[
\Tr(\rho_f)(\tau\Frob_p)-\Tr(\rho_f)(\Frob_p)=(\epsilon(\tau)-1)\epsilon(\Frob_p)a_p(f),
\]
so $T_p$ equals the image of $\frac{\Tr_{D^u}(\tau\Frob_p)-\Tr_{D^u}(\Frob_p)}{(\epsilon(\tau)-1)\epsilon(\Frob_p)}$.
\end{proof}

Let $R^\red = R \otimes_{R_{\bar D}} R_{\bar D}^\red$ be the quotient of $R$ that parameterizes reducibile deformations.

\begin{lemma}
\label{lem:Rred}
The homomorphism $R^\red \to \Z_p$ given by the reducible deformation $\epsilon \oplus 1$ is an isomorphism. 
\end{lemma}
\begin{proof}
The ring $R_{\bar D}^\red$ is the tensor product of deformation rings of $\omega$ and $1$, with universal deformation $\chi_\omega \oplus \chi_1$, where $\chi_\omega$ and $\chi_1$ are the universal deformations and $\omega$ and $1$, respectively.  Fixing the determinant to be $\epsilon$ gives $\chi_\omega=\epsilon \chi_1^{-1}$ in $R^\red$, so it suffices to prove that $\chi_1=1$ in $R^\red$.
A deformation of $1$ factors through the maximal abelian pro-$p$ quotient of $G_{\Q,Np}$, which, by the Kronecker--Weber theorem, is the pro-$p$ quotient of $\Gal(\Q(\zeta_{N^\infty p^\infty})/\Q)$. Since $p \nmid (N-1)$, the maximal pro-$p$ quotient is unramified at $N$. The finite-flat condition forces $\chi_1$ to be unramified at $p$, so it must be unramified everywhere, and hence trivial.  
\end{proof}
Let $J=\ker(R \to R^\red)$ be the reducibility ideal of $R$; \cref{lem:Rred} implies that $R/J=\Z_p$.
Let $B=B_{\bar D,\mathfrak{c}}\otimes_{R_{\bar D,\mathfrak{c}}} R$ and $C=C_{\bar D,\mathfrak{c}}\otimes_{R_{\bar D,\mathfrak{c}}} R$. By \eqref{eq:BC=J}, there is a surjective map
\[
B \otimes_R C \onto J.
\]
\begin{lemma}
\label{lem:cyclicity}
The $R$-modules $B$ and $C$ are cyclic and $J$ is a principal ideal.
\end{lemma}
\begin{proof}
By \cite[Theorem 4.3.5]{WWE4} (which is the analog of \eqref{eq:B dual} with deformation conditions), there are isomorphisms
\[
\Hom_{R}(B, \F_p) \cong H^1_{\mathfrak{c}}(G_{\Q,Np},\F_p(1)),  \quad \Hom_{R}(C, \F_p)=H^1_{\mathfrak{c}}(G_{\Q,Np},\F_p(-1)).
\]
These groups have been computed to be one-dimensional in \cite[Proposition 6.3.2 and Lemma 6.3.6]{WWE20}, respectively\footnote{See also \cite{CE05}, especially Lemma 3.9 and Proposition 5.4,  for an earlier proof of the same result, in slightly different terms.}.   Since $\Hom_{R}(B, \F_p)$ and $\Hom_{R}(C, \F_p)$ are one dimensional,  Nakayama's lemma implies that $B$ and $C$ are cyclic $R$-modules. Then the surjection
\[
B \otimes_R C \onto J
\]
of \eqref{eq:BC=J} implies that $J$ is principal.
\end{proof}

\begin{proposition}
\label{prop:structure of R}
There is an isomorphism
\[
R/(p,\m_{\bar D}^2) \isoto \F_p[\epsilon]/(\epsilon^2)
\]
given by a pseudorepresentation $D_{b,c,\phi}$ of the form \eqref{eq:bc tangent} with $\chi_1=\omega$ and $\chi_2=1$,  where $b$ and $c$ are cocycles representing generators of the groups $H^1_\mathfrak{c}(G_{\Q,Np},\F_p(1))$ and $H^1_\mathfrak{c}(G_{\Q,Np},\F_p(-1))$, respectively.
\end{proposition}
\begin{proof}
By \cref{lem:cyclicity},  there is an element $x \in J$ that generates $J$. By \cref{lem:Rred}, $R/J =\Z_p$.  This implies that $\m_{\bar D}=(p,x)$, and that the maximal ideal of $R/pR$ is principal.  There is a surjection $R \onto \bT$ by \cref{lem:R to T}, and $\bT$ is a free $\Z_p$-module of rank at least $2$ by \cite[Theorem B]{us}, so $R/pR \ne \F_p$.  Hence $R/(p,\m_{\bar D}^2)$ is isomorphic to $\F_p[\epsilon]/(\epsilon^2)$.  This isomorphism defines an element of the tangent space $\mathfrak{t}_{\bar D}$ of $R_{\bar D}$. This element cannot be a reducible deformation by \cref{lem:Rred}, so it must be of the claimed form.
\end{proof}

Let $R^{\mathrm{pseudo-min}}$ be the quotient of $R$ by the ideal generated by $\Tr_{D^u}(\sigma)-2$ for all $\sigma \in I_N$. This is called the \emph{pseudo-minimal} quotient because it is imposing the condition that the pseudorepresentation equals the trivial pseudorepresentation on inertia at $N$.  A pseudorepresentation is called \emph{minimal} if it comes from a Cayley--Hamilton representation $(E,D,\rho)$ such that $\rho|_{I_N}=1$. Note that a pseudo-minimal pseudorepresentation may not be minimal: a Steinberg-at-$N$ representation is pseudo-minimal but not minimal.

Under the surjection $R \onto \bT$ of \cref{lem:R to T}, the quotient $R^\mathrm{pseudo-min}$ should correspond to quotient of $\bT$ that acts on forms of level $\Gamma_0(N)$.  Since $p \nmid (N-1)$, results of Mazur \cite[Proposition II.9.7]{Mazur77} imply that there are no cuspforms $f$ of weight 2 and level $\Gamma_0(N)$ such that $\bar\rho_f = \omega \oplus 1$.  Accordingly, if $R \cong \T$ then one should expect that $R^\mathrm{pseudo-min} = R^\red =\Z_p$.  Indeed, this is the case.

\begin{lemma}
\label{lem:nothing at level N}
The map $R^\mathrm{pseudo-min} \to \Z_p$ given by the deformation $\epsilon \oplus 1$ is an isomorphism. 
\end{lemma}
\begin{proof}
The deformation $\epsilon \oplus 1$ is obviously pseudo-minimal (in fact, minimal), so it defines a surjective homomorphism $R^\mathrm{pseudo-min} \onto \Z_p$.  To show it is an isomorphism, it is enough to show that the tangent space of $R^\mathrm{pseudo-min}/pR^\mathrm{pseudo-min}$ is trivial.  Since the tangent space of $R/pR$ is one-dimensional and generated by $D_{b,c,\phi}$ by \cref{prop:structure of R}, it is enough to show that $D_{b,c,\phi}$ is not pseudo-minimal.  Recall the formula \eqref{eq:bc tangent}
\[
D_{b,c,\phi}(x)=\omega(x)+1+\epsilon(\omega(x)\phi(x)+b(x)c(x)-\phi(x)).
\]
Since $\omega$ is unramified at $N$, for $\sigma \in I_N$ this equation simplifies to
\[
D_{b,c,\phi}(\sigma)=2+b(\sigma)c(\sigma)\epsilon.
\]
But the cocycles $b$ and $c$ must be ramified at $I_N$, so there is $\sigma \in I_N$ such that $b(\sigma)c(\sigma) \ne 0$.  This implies that $\epsilon$ is in the kernel of the map
\[
\F_p[\epsilon]/(\epsilon^2) \isoto	R/(p,\m^2) \onto R^\mathrm{pseudo-min}/(p,\m^2),
\]
completing the proof.
\end{proof}

\section{Computation of the local deformation ring}
\label{sec:local}
In this section we define a local deformation ring $R_N$ that is naturally augmented over $\Z_p$ with augmentation ideal $I$.  The global deformation ring $R$ is an $R_N$-algebra in a natural way and the extension $IR$ of $I$ to $R$ is the kernel of the map $R \to R^\mathrm{pseudo-min}$.  In particular, \cref{lem:nothing at level N} implies that $R/IR=\Z_p$.  This says that the global deformations are completely controlled by the local deformations (indeed, by Nakayama's lemma, it says that $R$ is a cyclic $R_N$-module).   Finally, we completely characterize the local deformations, proving that they all come from inducing a character of $G_{\Q_{N^2}}$, and thereby prove an isomorphism $R_N \isoto \Lambda^+$. 
\subsection{The deformation ring of the supercuspidal character}
One way to construct a deformation $\rho : G_N \to \GL_2(A)$ of $\overline{D}|_{G_N}$ with unramified determinant is to induce a character from $G_{N^2}$.   As a preliminary to considering such inductions, we recall some properties of the universal such character.

Let $\tilde{\Lambda}$ be the universal deformation ring of the trivial character $G_{N^2}\to \F_p$,  where $G_{N^2}=\Gal(\bar{\Q}_N/\Q_{N^2})$.  By \cite[Section 1.4]{MazurDef89}, there is an isomorphism
\[
\tilde{\Lambda} = \Z_p \lb G_{N^2}^{\mathrm{ab, pro-}p} \rb,
\]
where $G_{N^2}^{\mathrm{ab, pro-}p}$ is the maximal abelian pro-$p$ quotient and the universal character is the tautalogical one.  Fix a choice of Frobenius element $\Frob_{N^2} \in G_{N^2}$.  The local Artin map induces an isomorphism $G_{N^2}^{\mathrm{ab, pro-}p}\cong \Q_{N^2}^{\times,\mathrm{pro-}p}$ that sends $\Frob_{N^2}$ to $N$.  Let $\Lambda$ denote the quotient of $\tilde{\Lambda}$ given by identifying $\Frob_{N^2}$ with $-N$.  Using the local Artin isomorphism as an identification, 
$\Lambda$ is identified with $\Z_p[\Delta]$, where $\Delta=\Z_{N^2}^{\times,\mathrm{pro-}p} = \F_{N^2}^{\times,\mathrm{pro-}p}$ is a cyclic group of order $p^r$, where $r=v_p(N+1)$.  (Here $v_p$ is the $p$-adic valuation normalized such that $v_p(p) = 1$.)  Denote the universal character $G_{N^2} \to \Lambda^\times$ by $[-]$.  Let $\delta \in \Delta$ be a generator.

Consider the Galois representation
\begin{equation}
\label{eq:rhoN}
\rho_N: G_N \to \GL_2(\Lambda)
\end{equation}
given by $\Ind_{G_{N^2}}^{G_N} [-]$.  This is a deformation of $\bar D|_{G_N}$ and it satisfies $\det(\rho_N)=\epsilon$.  For $\sigma \in I_N$, the trace of $\rho_N(\sigma)$ is given by
\begin{equation}
\label{eq:Tr(rhoN)}
\Tr(\rho_N(\sigma))=[\sigma]+[\sigma]^{-1}.
\end{equation}
This lands in the subring $\Lambda^+ \subset \Lambda$ fixed by the involution $\iota$ that acts as inversion on group-like elements.  For later use, we recall the structure of the ring $\Lambda^+$. 

\begin{lemma}
\label{lem:Lam+}
There is an isomorphism $\frac{\Z_p[x]}{(x \Psi(x))} \isoto \Lambda^+$ given by $x \mapsto [\delta]+[\delta^{-1}]-2$, where $\Psi(x)$ is a distinguished polynomial of degree $\frac{p^r-1}{2}$ with $v_p(\Psi(0))=r$.
\end{lemma}
\begin{proof}
First note that $\Lambda^+$ is equal to the subring of $\Lambda$ generated by $[\delta]+[\delta^{-1}]$. Indeed, every element of $\Lambda^+$ can be represented by a symmetric polynomial in $[\delta]$ and $[\delta^{-1}]$, and every such polynomial is a polynomial in $[\delta]+[\delta^{-1}]$.

Next note that the map
\begin{equation}
\label{eq:Lam+ fiber}
\Lambda \to \Z_p \times \prod_{i=1}^r \Z_p[\zeta_{p^i}]
\end{equation}
sending $[\delta]$ to $(1,\zeta_p,\dots,\zeta_{p^r})$ is injective with $p$-torsion cokernel. Taking $\iota$-fixed parts gives a map
\[
\Lambda^+ \to \Z_p \times \prod_{i=1}^r \Z_p[\zeta_{p^i}+\zeta_{p^i}^{-1}],
\]
again injective with $p$-torsion cokernel.  Hence the surjective map $\Z_p[x] \onto \Lambda^+$ given by $x \mapsto [\delta]+[\delta^{-1}]-2$ sends $x\Psi(x)$ to zero, where $\Psi(x)$ is the product of the minimal polynomials $\Psi_i(x)$ of $\zeta_{p^i}+\zeta_{p^i}-2$.  The induced map $\Z_p[x]/(x\Psi(x)) \onto \Lambda^+$ is a surjective homomorphism of free $\Z_p$-modules of the same finite rank, so it is an isomorphism.  Since each ring $\Z_p[\zeta_{p^i}+\zeta_{p^i}^{-1}]$ is totally ramified over $\Z_p$,  the polynomials $\Psi_i(x)$ are Eisenstein, so $v_p(\Psi(0))=r$. 
\end{proof}

\subsection{A computation of an inertial pseudodeformation ring}	
\label{subsec:R_N}
We now define a kind of local deformation ring $R_N$. Roughly speaking, it is the ring parameterizing ``deformations on inertia that extend to the decomposition group".  The main result of this section 
is \cref{prop:R_N=Lam}, which states that all inertia deformations that extend to the decomposition group are supercuspidal, in the sense that they arise from an induction construction. 

We first recall some properties of local Galois groups. There is an exact sequence
\[
0 \to I_N \to G_N \to \Gal(\Q_N^\mathrm{nr}/\Q_N) \to 0,
\] 
where $\Q_N^\mathrm{nr}$ is the maximal unramified extension. The group $\Gal(\Q_N^\mathrm{nr}/\Q_N)$ is isomorphic to $G_{\F_N}$ and hence is topologically generated by $\Frob_N \in G_N$. The group $I_N$ is complicated, but its maximal pro-$p$ quotient $I_N^{(p)}$ is pro-cyclic. Let $\tau \in I_N$ be a element that topologically generates $I_N^{(p)}$.  Frobenius acts on $\tau$ by
\[
\Frob_N \tau \Frob_N^{-1} = \tau^N.
\]

If $\rho$ is a representation of $I_N$ that extends to a representation of $G_N$, then $\rho(\tau)$ and $\rho(\tau^N)$ must be conjugate and thus have the same traces and determinants.  This motivates the following definition.

\begin{definition}
\label{defn:RN}
Let $R_N$ be the quotient of $R_{\bar D|_{I_N}}$ by the ideal generated by:
\begin{itemize}
\item $D^u(\sigma)-1$ for all $\sigma \in I_N$, and 
\item $\Tr_{D^u}(\tau)-\Tr_{D^u}(\tau^N)$.
\end{itemize} 
The pseudorepresentation associated to the trivial representation $I_N \to \GL_2(\Z_p)$ defines a map $R_N \onto \Z_p$, making $R_N$ into an augmented $\Z_p$-algebra.  Let $I=\ker(R_N \to \Z_p)$ be the augmentation ideal; explicitly, it is the ideal generated by $\Tr_{D^u}(\sigma)-2$ for all $\sigma \in I_N$.
\end{definition}

Of course, if $\rho: G_N \to A$ is a deformation of $\bar D|_{G_N}$ with unramified determinant, then $\rho|_{I_N}$ defines a map $R_N \to A$.  Thus restricting the universal pseudodeformation $D^u: G_{\Q, Np} \to R$ to $I_N$ induces a ring homormorphism $R_N \to R$.  The following lemma shows we are in the unusual situation that this map is surjective.

\begin{lemma}
\label{lem:SurjectiveToR}
The natural map $R_N \to R$ is surjective.
\end{lemma}

\begin{proof}
Let $I \subset R_N$ be the augmentation ideal of $R_N$ as in \cref{defn:RN}.  The ideal $IR$ is generated by $\Tr_{D^u}(\sigma)-2$ for all $\sigma \in I_N$, which is exactly the kernel of $R \to R^\mathrm{pseudo-min}$, so $R/IR=R^\mathrm{pseudo-min}$. By \cref{lem:nothing at level N}, this implies $R/IR=\Z_p$.  Then by Nakayama's lemma, the map $R_N \to R$ is surjective.
\end{proof}

There is a quotient of $R_N$ that parameterizes supercuspidal (that is, induced) deformations. Indeed, the pseudorepresentation $\rho_N: G_N \to \GL_2(\Lambda)$ constructed in the previous section is the universal induced representation.  By \eqref{eq:Tr(rhoN)}, its pseudorepresentation on inertia has values in the subring $\Lambda^+$ of $\Lambda$. This defines a surjective homomorphism $
R_N \onto \Lambda^+$.  The following proposition shows that, in fact, all deformations are supercuspidal.  (Note that such deformations are allowed to be reducible; for instance $1 \oplus \epsilon$ is supercuspidal as it is the induction of the trivial character.)
\begin{proposition}
\label{prop:R_N=Lam}
The map $R_N \onto \Lambda^+$ is an isomorphism of augmented $\Z_p$-algebras.
\end{proposition}
\begin{proof}
Let $\tilde{R}_N$ be the quotient of $R_{\bar D|_{I_N}}$ by the ideal generated by $D^u(\sigma)-1$ for all $\sigma \in I_N$.  That is, $\tilde{R}_N$ is the universal deformation ring of the trivial 2-dimensional pseudorepresentation on $I_N$ having trivial determinant.  Let $\tilde{D}: I_N \to \tilde{R}_N$ be the universal deformation.  Consider the representation
\[
\tilde{\rho}_N: I_N \to \GL_2(\Z_p\lb x \rb)
\]
obtained as the composite
\[
I_N \onto I_N^{(p)} =  \dia{\tau} \xrightarrow{\tau \mapsto \sm{1+x}{1}{x}{1}} \GL_2(\Z_p\lb x \rb).
\]
The pseudorepresentation of $\tilde{\rho}_N$ defines a map
\[
\psi:\tilde{R}_N \to \Z_p \lb x \rb
\]
We claim that $\psi$ is an isomorphism, with inverse given by the map
\[
\phi:\Z_p \lb x \rb \to \tilde{R}_N, \ x \mapsto \Tr_{\tilde{D}}(\tau)-2.
\]
Since $\psi(\Tr_{\tilde{D}}(\tau))=\Tr(\tilde{\rho}_N(\tau))=x+2$,  the composition $\psi \circ \phi$ is the identity. On the other hand, the map
\[
\phi \circ \psi: \tilde{R}_N \to \tilde{R}_N
\]
defines a pseudorepresentation $\tilde{D}': I_N \to \tilde{R}_N$.  To see that $\phi \circ \psi$ is the identity, it is enough to show that $\tilde{D}'=\tilde{D}$.  Since they both have trivial determinant, it suffices to show $\Tr_{\tilde{D}'}=\Tr_{\tilde{D}}$. By construction, 
\[
\Tr_{\tilde{D}'}(\tau)=\Tr_{\tilde{D}}(\tau).
\]
Then, by the pseudorepresentation identity \cref{item: pseudorep ident}, this implies that for all $n$,
\[
\Tr_{\tilde{D}'}(\tau^n)=\Tr_{\tilde{D}}(\tau^n).
\]
Since $\Tr_{\tilde{D}'}$ and $\Tr_{\tilde{D}}$ are continuous and agree on a dense subgroup of $\dia{\tau}$, they are equal on $\dia{\tau}$ (see \cref{rem:continuity}).  Finally, both $\Tr_{\tilde{D}'}$ and $\Tr_{\tilde{D}}$ send every element $\sigma \in \ker(I_N \to I_N^{(p)})$ to $2$. Indeed, for any Cayley--Hamilton representation $\rho: I_N \to E^\times$ inducing either one,  since $\rho \equiv 1 \pmod{\m E}$, the image of $\rho$ is pro-$p$, so $\rho$ factors through $I_N^{(p)}$. Since $\Tr_{\tilde{D}'}$ and $\Tr_{\tilde{D}}$ agree on $\dia{\tau}$ and on $\ker(I_N \to I_N^{(p)})$,  the pseudorepresentation identity implies that they agree on $I_N$.

Now, since $R_N$ is the quotient of $\tilde{R}_N$ by the relation $\Tr_{\tilde{D}}(\tau) = \Tr_{\tilde{D}}(\tau^N)$,  the map $\psi$ induces an isomorphism $R_N \cong \Z_p \lb x \rb /(f(x))$, where $f(x)= \Tr(\tilde{\rho}_N(\tau)) - \Tr(\tilde{\rho}_N(\tau^N))$.  
To compute $f(x)$ more explicitly, it is convenient to pass to an overring of $\Z_p \lb x\rb$ that contains the eigenvalues of $\tilde{\rho}_N(\tau)$, which are the roots of the polynomial $\lambda^2-(2+x)\lambda+1$. 
 Over the ring 
\[
\frac{\Z_p\lb x \rb[\lambda]}{\left(\lambda^2-(2+x)\lambda+1\right)} =\frac{\Z_p\lb x \rb[\lambda]}{\left((\lambda-1)^2-x\lambda\right)} \cong \Z_p\lb \lambda-1 \rb, \ x \mapsto\frac{(\lambda-1)^2}{\lambda},
\]
the eigenvalues\footnote{In fact, 
$\tilde{\rho}_N(\tau)$ is conjugate to the matrix $\sm{\lambda}{1}{0}{\lambda^{-1}}$ in $\GL_2(\Z_p\lb \lambda-1 \rb)$.} of $\tilde{\rho}_N(\tau)$ are $\lambda$ and $\lambda^{-1}$, so $f(x)=\lambda+\lambda^{-1}-(\lambda^N+\lambda^{-N})$.  Since $\Z_p\lb \lambda -1 \rb / \Z_p \lb x \rb$ is a torsion-free $\Z_p \lb x \rb$-module,  it follows that $\left(f(x) \Z_p\lb \lambda-1 \rb \right) \cap \Z_p\lb x\rb = f(x) \Z_p \lb x \rb$,  so that the map
\[
\frac{\Z_p\lb x \rb}{(f(x))} \to \frac{\Z_p \lb \lambda-1 \rb}{(f(x))}
\]
induced by $x \mapsto \frac{(\lambda-1)^2}{\lambda}$ is injective.
This shows that  $R_N$ is isomorphic to the subring generated by $\frac{(\lambda-1)^2}{\lambda}$ in the ring $A$ defined by
\[
A=\frac{\Z_p\lb \lambda-1 \rb}{\left(\lambda+\lambda^{-1}-\lambda^N-\lambda^{-N}\right)}.
\]
This presentation of $A$ can be simplified by factoring:
\[
\lambda+\lambda^{-1}-\lambda^N-\lambda^{-N} = -\lambda^{-N}(\lambda^{N+1}-1)(\lambda^{N-1}-1).
\]
Noting that $-\lambda^{-N} \in \Z_p\lb \lambda-1\rb$ is a unit and that, since $p \nmid(N-1)$ and $p^r \mid\mid (N+1)$ the ratios 
\[
\frac{\lambda^{N-1}-1}{\lambda-1},  \frac{\lambda^{N+1}-1}{\lambda^{p^r}-1}\in \Z_p\lb \lambda-1\rb
\]
are units, the ring $A$ can be written as
\[
A=\frac{\Z_p\lb \lambda-1 \rb}{(\lambda-1)(\lambda^{p^r}-1)}.
\]
Note that there is a surjective homomorphism $A \to \Lambda$ given by $\lambda \mapsto [\delta]$, which gives the presentation $\Lambda \cong \Z_p \lb \lambda-1 \rb/(\lambda^{p^r}-1)$. In this presentation, $\Lambda^+$ is the subring generated by $[\delta]+[\delta^{-1}]-2 = \frac{(\lambda-1)^2}{\lambda}$. 
Factoring $\lambda^{p^r}-1$ into a product of irreducible polynomials yields an embedding
\[
A \hookrightarrow \frac{\Z_p[\lambda-1]}{(\lambda-1)^2} \oplus \bigoplus_{i=1}^r \Z_p[\zeta_{p^i}],
\]
with $\lambda$ mapping to $\zeta_{p^i}$ in the rightmost factors.  This map induces the map \eqref{eq:Lam+ fiber} on the quotient $\Lambda$ of $A$.
Hence there is a commutative diagram
\[
\xymatrix{
R_N \ar@{->>}[d] \ar@{^(->}[r] & A \ar@{->>}[d] \ar@{^(->}[r] & \frac{\Z_p[\lambda-1]}{(\lambda-1)^2} \oplus \bigoplus_{i=1}^r \Z_p[\zeta_{p^r}] \ar@{->>}[d] \\
\Lambda^+ \ar@{^(->}[r] & \Lambda \ar@{->>}[r] & \Z_p \oplus \bigoplus_{i=1}^r \Z_p[\zeta_{p^r}]
}
\]
in which the leftmost vertical arrow is the map induced on the subrings generated by $\frac{(\lambda-1)^2}{\lambda}$ in middle vertical arrow. To see that the map $R_N \onto \Lambda^+$ is injective, it is enough to show that the kernel of the rightmost vertical arrow has trivial intersection with the subring generated by $\frac{(\lambda-1)^2}{\lambda}$. Since the kernel of this arrow is contained in the $\frac{\Z_p[\lambda-1]}{(\lambda-1)^2}$ factor, and $\frac{(\lambda-1)^2}{\lambda}$ maps to zero in this factor, this is clear.
\end{proof}

\section{Proof of \cref{thm:main}}
\label{sec:proof}
Combining \cref{lem:R to T}, \cref{lem:SurjectiveToR}, and \cref{prop:R_N=Lam}, we find that there is a chain of surjective ring homomorphisms
\begin{equation}
\label{eq:chain}
\Lambda^+ \isoto R_N \onto R \onto \bT.
\end{equation}
Letting $\pi:  \Lambda^+ \onto \bT$ denote the composition of these maps, we have a commutative diagram
\begin{equation}
\label{eq:LamT diagram}
\xymatrix{
\Lambda^+ \ar@{->>}[r]^{\pi} \ar@{->>}[dr] & \bT \ar@{->>}[d]^{\pi_\bT} \\
& \Z_p
}
\end{equation}
as in the set up of Wiles's numerical criterion \cite[Appendix]{FLT}, as improved by Lenstra \cite{lenstra} (see also \cite{dSRS}).  Let $J = \ker(\pi)$ be the augmentation ideal in $\Lambda^+$ and $I = \ker(\pi_\bT)$ the Eisenstein ideal in $\bT$. It is well known and easy to see that $\Ann_\bT(I)$ is the kernel of the quotient $\bT \to \bT^0$ of $\bT$ that acts faithfully on cuspforms. 

\begin{theorem}
The surjective maps in \eqref{eq:chain} are all isomorphisms.
\end{theorem}
\begin{proof}
Let $\eta =p^t$ be a generator of the ideal $\pi(\Ann_\bT(I)) \subseteq \Z_p$.  By the numerical criterion \cite[Criterion I]{dSRS}, it is enough to show $\#J/J^2 \le \eta$ .  It follows immediately from \cref{lem:Lam+}
that $\#J/J^2=p^r$.  Since $\Ann_\bT(I)=\ker(\bT\to \bT^0)$,  to show that $\eta \ge p^r$,  it is enough to show that the composite map 
\[
\bT \xrightarrow{\pi_\bT} \Z_p \to \Z/p^r\Z
\]
factors through $\bT^0$. In other words, it is enough to show that the Eisenstein series $E$ is a cuspform modulo $p^r$. This follows from \cite[Corollary 2.6]{us}, completing the proof.
\end{proof}

\begin{remark}
Since the ring $\Lambda^+$ is monogenic,  there is an alternative argument that does not use the numerical criterion (but still using the fact that $p^r \mid \eta$).  For this, note that the surjection $\pi$ and \cref{lem:Lam+} imply that $\bT$ has a presentation $\bT \cong \Z_p[x]/(xF(x))$ where $F(x)$ is a monic divisor of $\Psi(x)$. Then $\eta$ can be interpreted as the constant term $F(0)$ (up to a $p$-adic unit).  But since $p^r \mid \eta$, this implies that $\Psi(0) \mid F(0)$. Since $F(x) \mid \Psi(x)$, this implies $F(x)=\Psi(x)$, as desired.
\end{remark}

This completes the proof of \cref{thm:main}.

\section{Complement: relation to Massey products}
\label{sec:massey}
We have proven an isomorphism $\Lambda^+ \isoto R$ by identifying $R$ with a local deformation ring. Since the local deformation ring is so explicit, this gives us complete understanding of $R$, and in particular, its rank. In \cite{WWE20},  another method for studying the rank of $R$ is introduced, using obstructions in Galois cohomology that come from Massey products.  One might hope to use \cref{thm:main} and reverse the arguments of \cite{WWE20} to obtain non-trivial arithmetic results about vanishing of Massey products. 
In this section, we indicate how our results are related to Massey products, and conclude that the Massey products involved are particularly simple and not arithmetically interesting.  (This explains why the rank is basically constant in our case, rather than varying in an arithmetically interesting way as in \cite{WWE20}.)  We do not give another complete proof of \cref{thm:main} (although there is little doubt that one could give a proof along these lines).  Instead we attempt to illustrate why \cref{thm:main} is reasonable from the point of view of \cite{WWE20}.

\subsection{The strategy for relating Massey products to ranks}
It follows from \cref{prop:structure of R} that the tangent space of $R/pR$ is one-dimensional and spanned by a tangent vector $D_{b,c,\phi}$. So, there is an isomorphism $R/pR \cong \F_p \lb \epsilon\rb/(\epsilon^d)$ for some $d>1$ (including possibly $d=\infty$ if $R/pR \cong \F_p \lb \epsilon\rb$). This $d$ is then the $\F_p$-dimension of $R/pR$, which is an upper bound on the rank of $R$. One can study $d$ one step at a time: $d>2$ if and only if the tangent vector $R \onto \F_p[\epsilon]/(\epsilon^2)$ given by $D_{b,c,\phi}$ lifts to a map $R \onto \F_p[\epsilon]/(\epsilon^3)$. 
If $d>2$, then $d>3$ if and only if the map $R \onto \F_p[\epsilon]/(\epsilon^3)$ lifts to a map $R \onto \F_p[\epsilon]/(\epsilon^4)$, and so on. 

Given this interpretation of the $\F_p$-dimension of $R/pR$ in terms of lifts of the tangent vector, we now sketch how this is related to the vanishing of certain elements, called Massey products, in Galois cohomology $H^2(\Q,-)$.  For this, consider the problem of lifting $R \onto \F_p[\epsilon]/(\epsilon^2)$ to a map $R \onto \F_p[\epsilon]/(\epsilon^3)$.   Recall that $D_{b,c,\phi}$ comes from the trace of a representation $\rho_{b,c,\phi}$ \eqref{eq:rho bc tangent} that can be written as
\[
\rho_{b,c,\phi} = \ttmat{\omega(1+\phi_1 \epsilon)}{b+b_1\epsilon}{\omega c \epsilon}{1+\phi'_1 \epsilon} : G_\Q \to \GL_2(\F_p[\epsilon]/(\epsilon^2)),
\]
where $\phi_1$ is a choice of $1$-cochain satisfying $-d\phi_1=b \cup c$, $\phi'=bc-\phi$, and $b_1$ is a cochain satisfying $-db_1=b \cup \phi_1+\phi_1' \cup b$. A map $R \onto \F_p[\epsilon]/(\epsilon^3)$ lifting $D_{b,c,\psi}$ might come from a deformation $\rho_2$ of the form
\[
\rho_2 = \ttmat{\omega(1+\phi_1 \epsilon+\phi_2 \epsilon^2)}{b+b_1\epsilon+b_2 \epsilon^2}{\omega (c \epsilon+c_2\epsilon^2)}{1+\phi'_1 \epsilon+\phi_2'\epsilon^2} : G_\Q \to \GL_2(\F_p[\epsilon]/(\epsilon^3)).
\]
Here the functions $\phi_2$, $b_2$, $c_2$, and $\phi_2'$ are $1$-cochains\footnote{With coefficients in $\F_p$, $\F_p(1)$, $\F_p(-1)$, and $\F_p$, respectively.} that, in order for $\rho_2$ to be a homomorphism, must satisfy conditions on their coboundaries. For instance, $\phi_2$ satisfies
\[
-d\phi_2 = \phi_1 \cup \phi_1 + b \cup c_2 + b_1 \cup c.
\]
The right-hand side of this equation is a 2-cocycle, and the equation expresses the fact that this 2-cocycle is a 2-coboundary (i.e.~it vanishes in cohomology). A similar thing is true for the equations governing $b_2$, $c_2$, and $\phi_2'$. Conversely, without knowing that $\rho_2$ exists, if one knew that the relevant $2$-cocycles were $2$-coboundaries, one could define $\rho_2$ using these equations. The cohomology classes of these $2$-cocycles are examples of Massey products. This shows that Massey products are obstructions: the representation $\rho_2$ exists if and only if the Massey products vanish.

\begin{remark}
\label{rem:sketch}
We have glossed over several points in this sketch. First,  it is not clear that a pseudodeformation of $D_{b,c,\phi}$ must come from a true representation like $\rho_2$. This can be remedied by instead looking for deformations in the universal Cayley--Hamilton algebra.  
The results of Section \ref{sec:reduce to local} imply that this Cayley--Hamilton algebra is a GMA with a particularly simple form, so the true representations considered above are not too different from the universal case.  The main caveat is that, just as one doesn't need the cochain $b_1$ in order to define $D_{b,c,\phi}$, one also doesn't need the cochain $b_2$ in order to define the pseudorepresentation associated to $\rho_2$.
Second, in order for the deformation $\rho_2$ to define a map $R \onto \F_p[\epsilon]/(\epsilon^3)$, the pseudorepresentation must satisfy the local conditions required in the definition of $R$. This can be resolved by working with Galois cohomology with restricted ramification $H^2(G_{\Q,Np},-)$, and working carefully with the finite-flat condition.
\end{remark}
\subsection{Computation of the relevant Massey products}
The relevant Massey products can be computed explicitly in this case. There are two main reasons for this. First, the restriction maps
\[
\mathrm{res}_N: H^2(G_{\Q,Np},\F_p(i)) \to H^2(\Q_N,\F_p(i))
\]
for $i=0,1,-1$ are injective. This is a kind of local-to-global principle to compute the whether or not the global Massey-product classes vanish, it is enough to consider their restriction to local cohomology at $N$. Second, (and this is the most crucial difference with \cite{WWE20}), the cohomology group
$H^1(\Q_N,\F_p(1))$ is one-dimensional. Since $N \equiv -1 \pmod{p}$, there is an isomorphism $\F_p(1) \cong \F_p(-1)$ of $G_{\Q_N}$ modules, so $\mathrm{res}_N(b)$ and $\mathrm{res}_N(c)$ are both nonzero classes in the same one-dimensional space $H^1(\Q_N,\F_p(1))$. Up to rescaling, we can assume that $\mathrm{res}_N(b)=\mathrm{res}_N(c)$.

We will now sketch an argument that uses Massey products to explain why $\dim_{\F_p}(R/pR) \le \frac{p+1}{2}$ when $N \not \equiv -1 \pmod{p^2}$.
From this point on, we work exclusively with cohomology of $G_{\Q_N}$ and drop the $\mathrm{res}_N$ from the notation, which is justified by the local-to-global principle. Since $b=c$,  a particularly simple cochain $\phi$ satisfying $-d\phi=b\cup c = b \cup b$ is $\phi=\frac{1}{2}b^2$. Similarly, to find a cochain $b_1$ such that 
\[
-d b_1 = b\cup \phi + \phi' \cup b = b \cup \frac{1}{2}b^2 + \frac{1}{2}b^2\cup b
\]
take $b_1 = \frac{1}{6}b^3$. To simplify notation,  for $n<p$, let $b\up{n}=\frac{1}{n!}b^n$. Then $\rho_{b,c,\phi}$ takes the simple form
\[
\rho_{b,c,\phi} = \ttmat{\omega(1+b\up{2}\epsilon)}{b+b\up{3}\epsilon}{\omega b \epsilon}{1+b\up{2}\epsilon}.
\]
To deform this, one can take
\[
\rho_2 = \ttmat{\omega(1+b\up{2}\epsilon+b\up{4}\epsilon^2)}{b+b\up{3}\epsilon+b\up{5}\epsilon^2}{\omega (b \epsilon+b\up{3}\epsilon^2)}{1+b\up{2}\epsilon+b\up{4}\epsilon^2}.
\]
The obvious pattern continues: if $2n+1<p$, then there is a deformation $\rho_n: G_{\Q_N} \to \GL_2(\F_p[\epsilon]/(\epsilon^{n+1}))$ defined by
\[
\rho_n = \ttmat{ \omega(1+b\up{2}\epsilon+\dots +b\up{2n}\epsilon^n)}{b+b\up{2}\epsilon + \dots + b\up{2n+1}\epsilon^n}{\omega(b\epsilon+ \dots+b\up{2n-1}\epsilon^n)}{1+b\up2\epsilon+\dots+b\up{2n}\epsilon^n}.
\] 
Just as the pseudorepresentation $D_{b,c,\phi}$ does not require the cochain $b_1$ to be defined, the pseudorepresentation associated to $\rho_n$ does not require $b\up{2n+1}$ to be defined.  On other words, the pseudorepresentation associated to $\rho_n$ can be defined as long as $2n<p$. This defines a pseudorepresentation $D_\frac{p-1}{2}:G_N \to \F_p[\epsilon]/(\epsilon^\frac{p+1}{2})$. 
The obstruction to deforming $D_\frac{p-1}{2}$ is the $2$-cocycle
\begin{equation}
\label{eq:b^p}
\sum_{i=1}^{p-1} b\up{i} \cup b\up{p-i}.
\end{equation}
This is the Massey $p$th-power $\dia{b}^p$ of $b$, defined by Kraines \cite[Definition 11]{kraines}\footnote{Kraines actually only considers trivial coefficients,  but the generalization to non-trivial coefficients is straightforward.}.  By a variant of \cite[Theorem 14]{kraines} for non-trivial coefficients,  $\dia{b}^p$ is equal to $\partial(b)$ where $\partial$ is the connecting map in the exact sequence
\[
0 \to \F_p(1) \to (\Z/p^2\Z)(\omega) \to \F_p(1) \to 0,
\]
where $(\Z/p^2\Z)(\omega)$ is unramified $G_{\Q_N}$-module where $\Frob_N$ acts by $\omega(N)=-1$.  If $N \not\equiv -1 \pmod{p^2}$,  then a simple calculation shows that $\partial(b)=b \cup x$ for a non-trivial class $x \in H^1(\Q_N,\F_p)$.  By Tate duality, this implies that $\partial(b) \neq 0$.  
In other words, the $2$-cocycle \eqref{eq:b^p} is not a coboundary, and this gives an obstruction to deforming $D_\frac{p-1}{2}$.  If $\dim_{\F_p}(R/pR)$ were greater than $\frac{p+1}{2}$, there would be a surjective homomorphism $R \onto \F_p[\epsilon]/(\epsilon^\frac{p+3}{2})$,  from which one could construct a deformation of $D_\frac{p-1}{2}$, a contradiction.

The inequality $\dim_{\F_p}R/pR \le \frac{p+1}{2}$ goes most of the way to proving \cref{cor:main}. Indeed,  it is not difficult to show that every cuspform $f$ satisfying \eqref{eq:eis con} must be supercuspidal at $N$.  Taking the trace of the supercuspidal representation shows that the coefficient ring of $f$ contains $\Z_p[\zeta_p+\zeta_p^{-1}]$, so that $\mathrm{rank}_{\Z_p}(\bT) \ge \frac{p+1}{2}$.  Then $\dim_{\F_p}R/pR \le \frac{p+1}{2}$ and the surjection $R \onto \bT$ imply that these containments and inequalities are all equalities.  Of course, to complete the above sketch, one would have to deal with the issues mentioned in \cref{rem:sketch}.

\bibliography{eisenstein_congruences}
\bibliographystyle{alpha}
\end{document}